\documentclass[11pt,a4paper]{article}
\usepackage[left=2cm,right=2cm,top=2cm,bottom=2cm]{geometry}
\usepackage{graphicx}
\usepackage{amsmath,amssymb}
\numberwithin{figure}{subsection}
\usepackage{setspace}
\usepackage{float}
\usepackage{geometry}
\usepackage{caption}
\usepackage{listings}
\usepackage{amsthm}
\usepackage{hyperref}
\usepackage{tikz}
\usepackage{enumitem}

\newtheorem{theorem}{Theorem}[subsection]
\newtheorem{proposition}[theorem]{Proposition}
\newtheorem{corollary}[theorem]{Corollary}

\newtheorem{lemma}[theorem]{Lemma}
\newtheorem{definition}[theorem]{Definition}
\newtheorem{example}[theorem]{Example}

\usepackage{amsthm}

\title{\textbf{Rational Variant of Hartogs' Theorem\\ on Separate Holomorphicity}}
\author{Hanwen Liu}

\begin{document}
\maketitle
\begin{abstract}
    Given an uncountable algebraically closed field $K$, we proved that if partially defined function $f\colon K \times \dots \times K \dashrightarrow K$ defined on a Zariski open subset of the $n$-fold Cartesian product $K \times \dots \times K$ is rational in each coordinate whilst other coordinates are held constant, then $f$ is itself a rational function in $n$-variables.
\end{abstract}

\section{Introduction and Backgrounds}
We say {\em partial function} to mean partially defined function. For any partial function $f\colon A \dashrightarrow B$, we shall denote by $\operatorname{dom}f$ its {\em domain of definition}. For simplicity we introduce the following notation.
\begin{definition}\label{definition1_1}
    Let $A_1, \dots, A_n, B$ be arbitrary sets and $f\colon A_1 \times \dots \times A_n \dashrightarrow B$ a partial function. For each fixed $\left(a_1, \dots, a_n\right) \in A_1 \times \dots \times A_n$ and $i \in\{1, \dots, n\}$, define $f_{a_1, \dots, a_{i-1}}^{a_{i+1}, \dots, a_n}\colon A_i \dashrightarrow B$ to be the partial function satisfying the following conditions:
    \begin{enumerate}[itemsep=5pt]
        \item[$\mathrm{I}$)] $\operatorname{dom} f_{a_1, \dots, a_{i-1}}^{a_{i+1}, \dots, a_n}=\left\{a \in A_i \mid\left(a_1, \dots, a_{i-1}, a, a_{i+1}, \dots, a_n\right) \in \operatorname{dom}f\right\}$,
        \item[$\mathrm{II}$)] $f_{a_1, \dots, a_{i-1}}^{a_{i+1}, \dots, a_n}(a)=f\left(a_1, \dots, a_{i-1}, a, a_{i+1}, \dots, a_n\right)$ for all $\left(a_1, \dots, a_{i-1}, a, a_{i+1}, \dots, a_n\right) \in \operatorname{dom}f$.
    \end{enumerate}
\end{definition}
\begin{example}
    Consider the complex polynomial function $f\colon(z, w) \mapsto z^3+z w^3$ defined on $\mathbb{C} \times \mathbb{C}$. For any $a \in \mathbb{C}$, according to Definition \ref{definition1_1} we have that $f^a$ is the entire function $z \mapsto z^3+a^3 z$, and $f_a$ is the entire function $w \mapsto a w^3+a^3$.
\end{example}

\quad 

Throughout this article, varieties and regular mappings are in the sense of FAC. We shall recall the definition of a {\em rational map}.
\begin{definition}\label{definition1_2}
    Let $X$ and $Y$ be varieties over an algebraically closed field $k$, a partial function $f\colon X \dashrightarrow Y$ is termed a rational map, if $\operatorname{dom}f$ is a Zariski dense open subset of $X$ and $\left.f\right|_{\operatorname{dom}f}$ is a morphism of varieties over $k$.
\end{definition}
\begin{definition}\label{definition1_3}
    Let $X$ be a variety over an algebraically closed field $k$, a rational map $f\colon X \dashrightarrow k$ is also said to be a rational function on $X$.
\end{definition}

\quad

The main purpose of this note is to prove the following statement.
\begin{theorem}\label{theorem1_4}
    Let $k$ be an algebraically closed field of at least continuum cardinality and $f\colon k^n \dashrightarrow k$ be a partial function defined on a nonempty Zariski open subset of $k^n$, if $f_{a_1, \dots, a_{i-1}}^{a_{i+1}, \dots, a_n}\colon k \dashrightarrow k$ is a rational function for each $\left(a_1, \dots, a_n\right) \in \operatorname{dom}f$ and $i \in\{1, \dots, n\}$, then $f$ is a rational function.
\end{theorem}

\newpage
\section{Failure of the theorem in the case of countable cardinality}
In this section, by constructing an explicit counter-example, we show that the assertion in Theorem \ref{theorem1_4} does not hold for any ccountable algebraically closed field.
\begin{proposition}\label{proposition2_1}
    Let $k$ be a countable algebraically closed field, then there exists a function $f\colon k \times k \rightarrow k$ such that $f_a$ and $f^a$ are polynomial functions for all $a \in k$, but $f$ is not a rational function.
\end{proposition}
\begin{proof}
    Since $k$ is countably infinite, there exists a sequence $\left(a_n\right)_{n=0}^{\infty}$ of elements in $k$, such that the correspondence $n \mapsto a_n$ is a bijection between $k$ and $\mathbb{N}$. Define function $f\colon k \times k \rightarrow k$ by $f\left(a_n, a_m\right):=\sum_{i=0}^{n+m} \prod_{l=0}^i\left(a_n-a_l\right)\left(a_m-a_l\right)$. 
    
    Fix any $n \in \mathbb{N}$, then by construction $f^{a_m}(a)=\sum_{i=0}^m \prod_{n=0}^i\left(a-a_n\right)\left(a_m-a_n\right)$ for all $a \in k$ and in particular $f^{a_m}$ is a polynomial function of degree $m$. By symmetry of variables, $f_a$ is also a polynomial function for each $a \in k$.

    Provided that $f$ is a rational function, then by Hilbert's Nullstellensatz (c.f. Theorem 4.8 in \cite{1}) $f$ is a polynomial function since $\operatorname{dom} f=k^2$. But then $\operatorname{deg} f \geqslant \operatorname{deg} f^{a_n}=n \rightarrow \infty$ as $n \rightarrow \infty$, contradiction.
\end{proof}

\section{Proof of the theorem in the case of continuum cardinality}
In this section we prove that the assertion in Theorem \ref{theorem1_4} holds for algebraically closed fields of continuum cardinality. The absolute value of any complete non-Archimedean field is assumed tacitly to be non-trivial.
\begin{lemma}\label{lemma3_1}
    Let $X$ be a normal variety over an algebraically closed field $k$ and let $f\colon X \dashrightarrow k$ be a rational function, if for every point $p \in X \backslash \operatorname{dom} f$ there exists an algebraic curve $C$ on $X$ passing through $p$, such that $C \bigcap \operatorname{dom} f\neq \varnothing$ and $p$ is a removable singularity of the rational function $\left.f\right|_C\colon C \dashrightarrow k$, then there exists a regular function $h\colon X \rightarrow k$ such that $\left.h\right|_{\operatorname{dom} f} \equiv f|_{\operatorname{dom} f}$.
\end{lemma}
\begin{proof}
    Identify $k$ with the standard affine chart $\left\{[a: 1] \in k \mathbb{P}^1 \mid a \in k\right\}$ of $\mathbb{P}^1:=k \mathbb{P}^1$. Since $X$ is normal and $\mathbb{P}^1$ is proper, there exists 2 -codimensional subvarieties $Z_1, \dots, Z_n$ of $X$ and a rational map $\varphi\colon X \dashrightarrow \mathbb{P}^1$ defined on $X_0:=X \backslash \bigcup_{i=1}^n Z_i$ such that $\operatorname{dom} f \subseteq X_0$ and $\varphi(p)=[f(p): 1]$ for all $p \in \operatorname{dom} f$.

    By assumption, for any $p \in X_0$, there exists a Zariski open neighbourhood $U$ of $p$ in $X_0$ and an algebraic curve $C$ in $U$ passing through $p$, such that $[1: 0] \notin \varphi(C)$. Therefore $\varphi\left(X_0\right) \subseteq k$ and hence there exists rational function $g\colon X \dashrightarrow k$ such that the following diagram commutes:
    \begin{figure}[H]
        \centering 
        \begin{tikzpicture}[node distance = 1.5cm,scale=0.6]
            \node (X) at (0,0) {$X$};
            \node (P) at (2,0) {$\mathbb{P}^1$};
            \node (k) at (2,2) {$k$};
            \draw [dashed, thick, ->] (X)--(k);
            \draw [dashed, thick, ->] (X)--(P);
            \node at (1,-0.3) {$\varphi$};
            \node at (0.9,1.3) {$g$};
            \node at (2,0.9) {\rotatebox{90}{$\supseteq$}};
        \end{tikzpicture}
    \end{figure}

    Since $X$ is normal and $X \backslash \operatorname{dom}g =\bigcup_{i=1}^n Z_i$ is of codimension at least 2, by Hartogs' extension theorem we obtain that $Z_1, \dots, Z_n$ are removable singularities of $g$. Notice that by construction $\left.g\right|_{\operatorname{dom}f} \equiv f|_{\operatorname{dom} f}$, the statement is proved.
\end{proof}
\begin{corollary}\label{corollary3_2}
    Let $X, Y$ be normal varieties over an algebraically closed field $k$ and $f\colon X \times Y \dashrightarrow k$ be a partial function defined on a nonempty Zariski open subset of $X\times Y$, if $f_a, f^b$ are rational functions for all $(a, b) \in \operatorname{dom}f$ and there exists a rational function $g\colon X \times Y \dashrightarrow k$ such that $f(a, b)=g(a, b)$ for all $(a, b) \in \operatorname{dom} f \bigcap \operatorname{dom} g$, then $f$ is a rational function.
\end{corollary}
\begin{proof}
    The desired result follows immediately from Lemma \ref{lemma3_1}.
\end{proof}

\begin{proposition}\label{proposition3_3}
    Let $X_1, \dots, X_n$ be normal varieties over an algebraically closed field $k$ and $f\colon X_1 \times \dots \times X_n \dashrightarrow k$ be a partial function defined on a nonempty Zariski open subset of $X_1 \times \dots \times X_n$, if $f$ satisfies the following conditions:
    \begin{enumerate}[itemsep=5pt]
        \item[$\mathrm{I}$)] for each $\left(a_1, \dots, a_n\right) \in \operatorname{dom}f$ and $i \in\{1, \dots, n\},\quad f_{a_1, \dots, a_{i-1}}^{a_{i+1}, \dots, a_n}\colon X_i \dashrightarrow k$ is a rational function,
        \item[$\mathrm{II}$)] there exists Zariski dense subset $\Lambda_i$ of $X_i$ for each $i \in\{1, \dots, n\}$ and a rational function $g\colon X_1 \times \dots \times X_n \dashrightarrow k$, such that $f\left(a_1, \dots, a_n\right)=g\left(a_1, \dots, a_n\right)$ for all $\left(a_1, \dots, a_n\right) \in\left(\Lambda_1 \times \dots \times \Lambda_n\right) \bigcap$ $\operatorname{dom}f \bigcap \operatorname{dom}g$,
    \end{enumerate}
    then $f$ is a rational function.
\end{proposition}
\begin{proof}
    By induction it suffices to prove the statement for $n=2$.

    By virtue of Corollary \ref{corollary3_2} we can assume w.l.o.g. that $\operatorname{dom}f=\operatorname{dom}g$ and $X_1=\{a \in X_1 \mid \exists\ b \in X_2$ $(a,b)\left.\in\operatorname{dom}f\ \right\}$.

    For any $b \in \Lambda_2$, since $\Lambda_1$ is a Zariski dense subset of $X_1$ and $f^b(a)=g^b(a)$ for all $a\in \Lambda_1 \bigcap\operatorname{dom} f^b$, we obtain that $f^b \equiv g^b$.

    For any $a \in X_1$, since $\Lambda_2$ is a Zariski dense subset of $X_2$ and $f_a(b)=f^b(a)=g^b(a)=g_a(b)$ for all $b\in \Lambda_2 \bigcap \operatorname{dom}f_a$, we obtain that $f_a \equiv g_a$. Therefore we have $f \equiv g$.
\end{proof}
\begin{lemma}\label{lemma3_4}
    Let $k[t]$ be the polynomial ring over an arbitrary field $k$ and let $\mathbb{F}$ be a subfield of $k$, then $k(t) \bigcap \mathbb{F}((t))=\mathbb{F}(t)$.
\end{lemma}
\begin{proof}
    This is a direct consequence of Lemma 27.9 in \cite{2}.
\end{proof}
\begin{proposition}\label{proposition3_5}
    Let $n \in \mathbb{N}$ and $k\left[x_0, \dots, x_{n+1}\right]$ be the ring of polynomials in $n+2$ variables over an arbitrary field $k$, then $k\left(x_0, \dots, x_{n+1}\right)=\bigcap_{i=0}^{n+1} k\left(\left(x_0, \dots, x_{i-1}, x_{i+1}, \dots, x_{n+1}\right)\right)\left(x_i\right)$. 
\end{proposition}
\begin{proof}
    Define $K:=k\left(x_1, \dots, x_n\right)$ and define $L_i:=k\left(\left(x_0, \dots, x_{i-1}, x_{i+1}, \dots, x_{n+1}\right)\right)\left(x_i\right)$ for each $i=0, \dots, n+1$. For simplicity we denote $x:=x_0$ and $y:=x_{n+1}$.

    Since $L_i \subseteq k((y))\left(\left(x_0, \dots, x_{i-1}, x_{i+1}, x_n\right)\right)\left(x_i\right)$ for all $i \in\{0, \dots, n\}$, by induction we obtain $\bigcap_{i=0}^n L_i \subseteq k((y))\left(x_0, \dots, x_n\right) \subseteq k\left(x_0, \dots, x_n\right)((y))(x)=K((y))(x)$. Analogously we have $\bigcap_{i=1}^{n+1} L_i \subseteq K((x))(y) \subseteq K(y)((x))$. By Lemma \ref{lemma3_4} we conclude that $\bigcap_{i=0}^{n+1} L_i \subseteq K((y))(x) \bigcap K(y)((x))=(K(y))(x)=k\left(x_0, \dots, x_{n+1}\right)$.

    Conversely, since $k\left(x_0, \dots, x_{n+1}\right) \subseteq L_i$ for all $i \in\{0, \dots, n+1\}$, we obtain $k\left(x_0, \dots, x_{n+1}\right) \subseteq \bigcap_{i=0}^{n+1} L_i$.
\end{proof}
\begin{lemma}\label{lemma3_6}
    Let $k[t]$ be the polynomial ring over an arbitrary field $k$ and let $\sum_{i=0}^{\infty} a_i t^i \in k[[t]]$ be a formal power series with coefficients in $k$, then $\sum_{i=0}^{\infty} a_i t^i \in k(t)$ if and only if there exists $l, m \in \mathbb{N}$ such that the matrices 
    $$
    \left[\begin{array}{cccc}
    a_n & a_{n+1} & \dots & a_{n+m} \\
    a_{n+1} & a_{n+2} & \dots & a_{n+m+1} \\
    \vdots & \vdots & \ddots & \vdots \\
    a_{n+m} & a_{n+m+1} & \dots & a_{n+2 m}
    \end{array}\right]
    $$
    are degenerated for all integers $n \geqslant l$.
\end{lemma}
\begin{proof}
    This is a direct consequence of Lemma 5 of Chapter V.5 in \cite{3}.
\end{proof}
\begin{proposition}\label{proposition3_7}
    Let $A:=k\left\langle x_1, \dots, x_l\right\rangle$ be the $l$-dimensional Tate algebra over a complete non-Archimedean field $k$ and $\sum_{i=0}^{\infty} h_i t^i \in A[[t]]$ be a formal power series with coefficients in $A$,\quad if $\sum_{i=0}^{\infty} h_i\left(a_1, \dots, a_{l}\right) t^i \in k(t)$ for all $a_1, \dots, a_n \in\{a \in k:|a| \leqslant 1\}$ then $\sum_{i=0}^{\infty} h_i t^i \in \operatorname{Frac}(A[t])$.
\end{proposition}
\begin{proof}
    Denote by $\mathfrak{o}_k:=\{a \in k:|a| \leqslant 1\}$ the valuation ring of $k$. For any $n, m \in \mathbb{N}$ define the Hankel determinant by
    $$
    H_n^m:=\left|\begin{array}{cccc}
    h_n & h_{n+1} & \dots & h_{n+m} \\
    h_{n+1} & h_{n+2} & \dots & h_{n+m+1} \\
    \vdots & \vdots & \ddots & \vdots \\
    h_{n+m} & h_{n+m+1} & \dots & h_{n+2 m}
    \end{array}\right|
    $$
    and define $\Lambda_n^m:=\left\{\left(a_1, \dots, a_l\right) \in \mathfrak{o}_k^l \mid H_{n+i}^m\left(a_1, \dots, a_l\right)=0 \forall i \in \mathbb{N}\right\}$. Since $H_n^m \in A$ for all $n, m \in \mathbb{N}$, we obtain that $\Lambda_n^m$ are closed subsets of $\mathfrak{o}_k^l$ for all $(n, m) \in \mathbb{N}^2$. since $\sum_{i=0}^{\infty} h_i\left(a_1, \dots, a_l\right) t^i \in k(t)$ for all $\left(a_1, \dots, a_l\right) \in \mathfrak{o}_k^l$, by Lemma \ref{lemma3_6} $\mathfrak{o}_k^l=\bigcup_{n=0}^{\infty} \bigcup_{m=0}^{\infty} \Lambda_n^m$. Since $\mathfrak{o}_k^l$ is a complete metric space, by Baire's category theorem, there exists $s, r \in \mathbb{N}$ such that $\Lambda_s^r$ admits a nonempty interior. Therefore, for all $i \in \mathbb{N}$, by construction $H_{s+i}^r$ vanishes on a nonempty open subset of $\mathfrak{o}_k^l$ and hence is identically zero by the analytic continuation principle for restricted power series. Again by Lemma \ref{lemma3_6}, we conclude that $\sum_{i=0}^{\infty} h_i t^i \in(\operatorname{Frac} A)(t)=\operatorname{Frac}(A[t])$.
\end{proof}
\begin{proposition}\label{proposition3_8}
    Let $k\left[x_1, \dots, x_l\right]$ be the ring of polynomial in $l$ variables over a complete non-Archimedean field $k$ and let $\sum_{i_1=0}^{\infty} \dots \sum_{i_n=0}^{\infty} c_{i_1\dots i_n} x_1^{i_1} \dots x_l^{i_l} \in k\left\langle x_1, \dots, x_l\right\rangle$ be a power series converge on the $l$-fold cartesian product $\mathfrak{o}_k^l$ of the valuation ring $\mathfrak{o}_k$ of $k$, if 
    $$
    \sum_{i_1=0}^{\infty} \dots \sum_{i_n=0}^{\infty}\left(c_{i_1, \dots i_n} a_1^{i_1} \dots a_{n-1}^{i_{n-1}} a_{n+1}^{i_{n+1}} \dots a_{l}^{i_l}\right) x_n^{i_n} \in k\left(x_n\right)
    $$
    for each $\left(a_1, \dots, a_{l}\right) \in \mathfrak{o}_k^l$ and $n \in\{1, \dots, l\}$, then $\sum_{i_1=0}^{\infty} \dots \sum_{i_n=0}^{\infty}$ $c_{i_1, \dots i_n} x_1^{i_1} \dots x_l^{i_l} \in k\left(x_1, \dots, x_l\right)$.
\end{proposition}
\begin{proof}
    If $l=1$ then the statement holds trivially. For the subsequent proof we shall assume $l \geqslant 2$. By Proposition \ref{proposition3_7} we have that $\sum_{i_1=0}^{\infty} \dots \sum_{i_n=0}^{\infty} c_{i_1, \dots i_n} x_1^{i_1} \dots x_l^{i_l} \in \operatorname{Frac}\left(k\left\langle x_1, \dots, x_{n-1}, x_{n+1}, \dots, x_l\right\rangle\left[x_n\right]\right)$ $\subseteq k\left(\left(x_1, \dots, x_{n-1}, x_{n+1}, \dots, x_l\right)\right) \left(x_n\right)$ for all $n=1, \dots, l$. The desired result then follows from Proposition \ref{proposition3_5}.
\end{proof}
\begin{theorem}\label{theorem3_9}
    Let $k\left[x_1, \dots, x_n\right]$ be the ring of polynomials in $n$ variables over an algebraically closed complete non-Archimedean field $k$ and let $f\colon \mathbb{D}^n \rightarrow k$ be an arbitrary function defined on the $n$-fold cartesian product of the unit disc $\mathbb{D}:=\{a \in k:|a|<1\}$ in $k$, if for each $\left(a_1, \dots, a_n\right) \in \mathbb{D}^n$ and each $j \in\{1, \dots, n\}$ there exists a power series $\sum_{i=0}^{\infty} b_i x_j^i \in k\left\{x_j\right\}$ converges in $\mathbb{D}$ such that $f\left(a_1, \dots, a_{i-1}, a, a_{i+1}, \dots, a_n\right)=\sum_{i=0}^{\infty} b_i a^i$ for all $a \in \mathbb{D}$, then there exists a power series $\sum_{i_1=0}^{\infty} \dots$ $\sum_{i_n=0}^{\infty} c_{i_1 \dots i_n} x_1^{i_1} \dots x_n^{i_n} \in k\left\{x_1, \dots, x_n\right\}$ converge in $\mathbb{D}^n$ such that $f\left(a_1, \dots, a_n\right)=\sum_{i_1=0}^{\infty} \dots $ $\sum_{i_n=0}^{\infty} c_{i_1 \dots i_n}$ $a_1^{i_1} \dots a_n^{i_n}$ for all $\left(a_1, \dots, a_n\right) \in \mathbb{D}^n$.
\end{theorem}
\begin{proof}
    This is a direct consequence of Stawski's theorem in \cite{4}.
\end{proof}
\begin{lemma}\label{lemma3_10}
    Let $K$ be a complete non-Archimedean field and let $\mathcal{O}_K$ be its valuation ring, then
    \begin{enumerate}
        \item[$\dot{\mathrm{I}}$)] the cardinality of $\mathcal{O}_K$ is at least continuum,
        \item[$\ddot{\mathrm{II}}$)] $\mathcal{O}_K$ is a Zariski dense subset of $K$.
    \end{enumerate}
\end{lemma}
\begin{proof}
    Since the metric of $\mathcal{O}_K$ is by assumption nontrivial, by the general theory of topological groups, we conclude that $\mathcal{O}_K$ admits no isolated points. Therefore, $\mathcal{O}_K$ is a perfect complete metrizable space, and hence is of at least continuum cardinality by Proposition 6.6.5 in \cite{5}.

    The second assertion is a direct consequence of $\dot{\mathrm{I}}$).
\end{proof}
\begin{theorem}\label{theorem3_11}
    Let $K$ be an algebraically closed complete non-Archimedean field and let $f\colon K^n \dashrightarrow K$ be a partial function defined on a nonempty Zariski open subset of $K^n$, if $f_{a_1, \dots, a_{i-1}}^{a_{i+1}, \dots, a_n}\colon K \dashrightarrow K$ is a rational function for every $\left(a_1, \dots, a_n\right) \in \operatorname{dom}f$ and $i \in\{1, \dots, n\}$, then $f$ is a rational function.
\end{theorem}
\begin{proof}
    Denote by $\mathcal{O}_K$ the valuation ring of $K$ and denote by $K\left[x_1, \dots, x_n\right]$ the ring of polynomials in $n$ variables over $K$. Up to a homothety we can assume w.l.o.g. that $\mathcal{O}_K^{\oplus n} \subseteq \operatorname{dom}f$. By Theorem \ref{theorem3_9} there exists $\sum_{i_1=0}^{\infty} \dots \sum_{i_n=0}^{\infty} c_{i_1 \dots i_n} x_1^{i_1} \dots x_n^{i_n} \in K\left\langle x_1, \dots, x_n\right\rangle$ such that $f\left(a_1, \dots, a_n\right)=\sum_{i_1=0}^{\infty} \dots \sum_{i_n=0}^{\infty} c_{i_1\dots i_n} a_1^{i_1} \dots a_n^{i_n}$ for all $a_1, \dots, a_n \in \mathcal{O}_K$. By Proposition \ref{proposition3_8} and the assumption on $f$, we obtain that $\sum_{i_1=0}^{\infty} \dots \sum_{i_n=0}^{\infty} c_{i_1 \dots i_n} x_1^{i_1}$ $\dots x_n^{i_n} \in K\left(x_1, \dots, x_n\right)$. By Lemma \ref{lemma3_10} and Proposition \ref{proposition3_3} we conclude that $f$ is a rational function.
\end{proof}
\begin{corollary}\label{corollary3_12}
    Let $k$ be an algebraically closed field of continuum cardinality and $f\colon k^n \dashrightarrow k$ be a partial function defined on a nonempty Zariski open subset of $k^n$, if $f_{a_1, \dots, a_{i-1}}^{a_{i+1}, \dots, a_n}\colon k \dashrightarrow k$ is a rational function for every $\left(a_1, \dots, a_n\right) \in  \operatorname{dom}f$ and $i \in\{1, \dots, n\}$, then $f$ is a rational function.
\end{corollary}
\begin{proof}
    By Theorem \ref{theorem3_11} it suffices to prove that $k$ admits a complete non-Archimedean absolute value. Indeed, by the model theory of algebraically closed fields, if chark $=p>0$ then $k$ is isomorphic to the completion of an algebraic closure of the local field $\mathbb{F}_p((t))$, and if char $k=0$ then $k$ is isomorphic to the $p$-adic Tate field $\mathbb{C}_p$.
\end{proof}

\section{Proof of the theorem in the case of cardinality beyond continuum}
In this section we prove that the assertion in Theorem \ref{theorem1_4} holds for algebraically closed fields of cardinality exceeding continuum.

For any algebraically closed field $k$ and rational function $f\colon k \dashrightarrow k$, as usual we denote by $\operatorname{deg}f$ the mapping degree of $f$ and denote by $\operatorname{ord}_{\infty}f$ the order of $f$ at infinity.
\begin{lemma}\label{lemma4_1}
    Let $X$ be an algebraic subset of the $n$-dimensional affine space over an algebraically closed field $k$ and let $K$ be an algebraically closed subfield of $k$, then $X \bigcap K^n$ is a Zariski dense subset of $X$.
\end{lemma}
\begin{proof}
    Denote $m:=\operatorname{dim}_k X$ then by Noether's normalization lemma there exists a $k$-linear transform $T\colon k^n \longrightarrow k^m$ such that $f:=\left.T\right|_X$ is a surjective finite morphism. In particular, the $k$-linearity of $T$ yields $f^{-1}\left(K^m\right)=X \bigcap K^n$. Since $K$ is infinite as it is algebraically closed, and the Zariski topology on the affine line $k^1$ is cofinite, we obtain that $K^m$ is a Zariski dense subset of $k^m$.

    Since $k^m$ is normal, we have that $\phi$ is an open mapping. Take any nonempty Zariski open subset $U$ of $X$, then $f(U) \bigcap K^m \neq \varnothing$. Since $f$ is surjective, we have $U \bigcap\left(X \bigcap K^n\right)=f^{-1}\left(f(U) \bigcap K^m\right) \neq \varnothing$.
\end{proof}
\begin{proposition}\label{proposition4_2}
    Let $X$ be a variety over an uncountable algebraically closed field $k$ and let $\left(X_i\right)_{i=0}^{\infty}$ be a sequence of Zariski closed subsets of $X$, if $X \neq X_i$ for all $i \in \mathbb{N}$ then $X \neq \bigcup_{i=0}^{\infty} X_i$.
\end{proposition}
\begin{proof}
    By passing to an affine patch of $X$ and applying Noether's normalization lemma we can assume w.l.o.g. that $X=k^n$. Apply induction on $n \in \mathbb{N}$:
    
    If $n=0$ then the statement holds trivially. For $n=1$ we have that $X_i$ is finite for all $i \in \mathbb{N}$ and hence $\bigcup_{i=0}^{\infty} X_i$ is countable, while by assumption $k^1$ is not. 

    Suppose that $n \geqslant 2$. Since the Grassmannian $\operatorname{Gr}\left(n-1, k^n\right) \cong \mathbb{P}_{n-1}(k)$ is uncontable, by Dirichlet's Schubfachprinzip there exists a hyperplane $H$ in $k^n$ such that $H \neq X_i \bigcap H$ for all $i \in \mathbb{N}$. By the induction hypothesis we have $H \neq \bigcup_{i=0}^{\infty} X_i \bigcap H$. Therefore in particular $k^n \neq \bigcup_{i=0}^{\infty} X_i$.
\end{proof}
\begin{proposition}\label{proposition4_3}
    Let $X$ be a variety over an algebraically closed field $k$ of at least continuum cardinality and let $\left(\Lambda_i\right)_{i=0}^{\infty}$ be a sequence of subsets of $X$, if $X=\bigcup_{i=0}^{\infty} \Lambda_i$, then there exists an integer $n \geqslant 0$ and a Zariski dense subset $\Lambda$ of $X$, such that $\Lambda \subseteq \Lambda_n$ and $\Lambda$ is of at most continuum cardinality.
\end{proposition}
\begin{proof}
    By passing to a Zariski dense open subset of $X$, we can assume w.l.o.g. that $X$ is a subvariety of affine space $k^n$. By Lowenheim-Skolem theorem downward, there exists an algebraically closed subfield $K$ of $k$, such that $K$ is of continuum cardinality. By virture of Lemma \ref{lemma4_1}, we can assume w.l.o.g. that the cardinality of $k$ is continuum. Since $k$ is uncountable and algebraically closed, by Proposition \ref{proposition4_2} there exists $n \in \mathbb{N}$ such that $\Lambda_n$ is not Zariski nowhere-dense in $X$. Since $X$ is irreducible, we obtain that $\Lambda_n$ is a Zariski dense subset of $X$. Since $\Lambda_n \subseteq X \subseteq k^n$, and $k^n$ is equinumerous to $k$, we conclude that the cardinality of $\Lambda_n$ is at most continuum.
\end{proof}
\begin{lemma}\label{lemma4_4}
    Let $a, a_0, \dots, a_l$ be $l+2$ elements of an algebraically closed field $k$ and $P, Q: k \rightarrow k$ be polynomial functions, if $\operatorname{deg} P=n$ and $\operatorname{deg} Q=m=l-n$, then the determinant
    $$
    \Delta=\left|\begin{array}{cccccccc}
        1 & a & \dots & a^m & 0 & 0 & \dots & 0 \\ 
        P\left(a_0\right) & P\left(a_0\right) a_0 & \dots & P\left(a_0\right) a_0^m & Q\left(a_0\right) & Q\left(a_0\right) a_0 & \dots & Q\left(a_0\right) a_0^n \\ 
        \vdots & \vdots & \ddots & \vdots & \vdots & \vdots & \ddots & \vdots \\ 
        P\left(a_l\right) & P\left(a_l\right) a_l & \dots & P\left(a_l\right) a_l^m & Q\left(a_l\right) & Q\left(a_l\right) a_l & \dots & Q\left(a_l\right) a_l^n    
    \end{array}\right|
    $$
    is equal to $Q(a) \cdot \operatorname{res}(P, Q) \prod_{i=0}^{l-1} \prod_{j=i+1}^{l}\left(a_j-a_i\right)$.
\end{lemma}
\begin{proof}
    Define the following matrices
    $$
    \begin{aligned}
        A=&\left[\begin{array}{cccc}
        Q\left(a_0\right) & Q\left(a_0\right) a_0 & \dots & Q\left(a_0\right) a_0^n \\
        Q\left(a_1\right) & Q\left(a_1\right) a_1 & \dots & Q\left(a_1\right) a_1^n \\
        \vdots & \vdots & \ddots & \vdots \\
        Q\left(a_n\right) & Q\left(a_n\right) a_n & \dots & Q\left(a_n\right) a_n^n
        \end{array}\right]\\ 
        B=&\left[\begin{array}{cccccccc} 
            P\left(a_0\right)\left(a_0-a\right) & P\left(a_0\right)\left(a_0-a\right) a_0 & \dots &P\left(a_0\right)\left(a_0-a\right) a_0^{m-1} & Q\left(a_0\right) & Q\left(a_0\right)a_0 & \dots & Q\left(a_0\right) a_0^n \\ 
            P\left(a_1\right)\left(a_1-a\right) & P\left(a_1\right)\left(a_1-a\right) a_1 & \dots &P\left(a_1\right)\left(a_1-a\right) a_1^{m-1} & Q\left(a_1\right) & Q\left(a_1\right)a_1 & \dots & Q\left(a_1\right) a_1^n \\ 
            \vdots & \vdots & \ddots & \vdots & \vdots & \vdots & \ddots & \vdots \\ 
            P\left(a_l\right)\left(a_l-a\right) & P\left(a_l\right)\left(a_l-a\right) a_l & \dots & P\left(a_l\right)\left(a_l-a\right) a_l^{m-1} & Q\left(a_l\right) & Q\left(a_l\right)a_l & \dots & Q\left(a_l\right) a_l^n
        \end{array}\right]
    \end{aligned}
    $$

    If $m=0$ i.e. $Q$ is a constant function, then $\Delta=\operatorname{det} A=Q\left(a_0\right)^{n+1} \prod_{i=0}^{n-1} \prod_{j=i+1}^n\left(a_j-a_i\right)=Q(a) \cdot \operatorname{res}(P, Q) \prod_{i=0}^{l-1} \prod_{j=i+1}^l\left(a_j-a_i\right)$. Suppose that $m \geqslant 1$, then by applying elementary column transformations, we obtain that $\Delta= \operatorname{det}B$. Define linear polynomial function $L_a: k \rightarrow k$ by $L_a(b)=b-a$, then by observation $B$ is the product of the Vandermonde matrix $(a_j^i)_{j=0,\dots,l}^{i=0,\dots,l}$ with the Sylvester matrix of $L_a P$ and $Q$. By the universal property of the resultant, we have $\operatorname{res}(L_aP, Q)=\operatorname{res}(P, Q) \operatorname{res}\left(L_a, Q\right)=Q(a) \cdot \operatorname{res}(P, Q)$. Therefore we conclude that $\operatorname{det} B=Q(a) \cdot \operatorname{res}(P, Q) \prod_{i=0}^{l-1} \prod_{j=i+1}^l\left(a_j-a_i\right)$.
\end{proof}
\begin{proposition}\label{proposition4_5}
    Let $k$ be an algebraically closed field and $f\colon k \dashrightarrow k$ be a nonzero rational function, then the following properties hold:
    \begin{enumerate}
        \item[$\mathrm{I}$)] $n:=\operatorname{deg} f+\min \left\{0, \operatorname{ord}_{\infty} f\right\}$ and $m:=\operatorname{deg} f-\max \left\{0, \operatorname{ord}_{\infty} f\right\}$ are non-negative integers,
        \item[$\mathrm{II}$)] for any $l+2$ elements $a, a_0, \dots, a_l$ of $\operatorname{dom}f$, if $a_0, \dots, a_l$ are pairwise distinct and $l=n+m$, then the determinants
        $$
        \begin{aligned}
            \alpha=&\left|\begin{array}{cccccccc}
                1 & a & \dots & a^n & 0 & 0 & \dots & 0 \\ 
                1 & a_0 & \dots & a_0^n & f\left(a_0\right) & f\left(a_0\right)a_0 & \dots & f\left(a_0\right) a_0^m\\ 
                \vdots & \vdots & \ddots & \vdots & \vdots & \vdots & \ddots & \vdots \\ 
                1 & a_l & \dots & a_l^n & f\left(a_l\right) & f\left(a_l\right)a_l & \dots & f\left(a_l\right) a_l^m
            \end{array}\right|\\ 
            \beta=&\left|\begin{array}{cccccccc}
                1 & a & \dots & a^n & 0 & 0 & \dots & 0 \\ 
                f\left(a_0\right) & f\left(a_0\right)a_0 & \dots & f\left(a_0\right)a_0^m & 1 & a_0 & \dots & a_0^n \\ 
                \vdots & \vdots & \ddots & \vdots & \vdots & \vdots & \ddots & \vdots \\ 
                f\left(a_l\right) & f\left(a_l\right)a_l & \dots & f\left(a_l\right)a_l^m & 1 & a_l & \dots & a_l^n
            \end{array}\right|
        \end{aligned}
        $$
        satisfy that $\beta \neq 0$ and $f(a)=(-1)^{n m} \alpha / \beta$.
    \end{enumerate}
\end{proposition}
\begin{proof}
    By the assumption on $f$, there exists coprime polynomial function $P, Q\colon k \rightarrow k$ such that $Q(a) \neq 0$ and $f=\dfrac{P(a)}{Q(a)}$ for all $a \in \operatorname{dom}f$. By definition we have $n=\operatorname{deg} P \geqslant 0$ and $m=\operatorname{deg} Q \geqslant 0$. 
    
    By Lemma \ref{lemma4_4}, we have $\alpha \cdot \prod_{i=0}^{l} Q\left(a_i\right)=P(a) \cdot \operatorname{res}(Q, P) \prod_{i=0}^{l-1} \prod_{j=i+1}^l\left(a_j-a_i\right)$ and $\beta \cdot \prod_{i=0}^{l} Q\left(a_i\right)=Q(a) \cdot \operatorname{res}(Q, P) \prod_{i=0}^{l-1} \prod_{j=i+1}^l\left(a_j-a_i\right)$. Since $a_0, \dots, a_{l}$ are distinct and $P$ is coprime to $Q$, we obtain that $\prod_{i=0}^{l-1} \prod_{j=i+1}^l\left(a_j-a_i\right) \neq 0$ and $\operatorname{res}(P, Q) \neq 0$. Since $a, a_0, \dots, a_l \in \operatorname{dom} f$, we have that $Q(a) \neq 0$, and $Q\left(a_i\right) \neq 0$ for all $i=0, \dots, l$. Therefore $\beta \neq 0$ and $\dfrac{\alpha}{\beta}=\dfrac{\operatorname{res}(Q, P)}{\operatorname{res}(P, Q)} \dfrac{P(a)}{Q(a)}=(-1)^{n m} f(a)$.
\end{proof}
\begin{theorem}\label{theorem4_6}
    Let $X$ be a normal variety over an algebraically closed field $k$ of cardinality exceeding continuum and let $f\colon X \times k^1 \dashrightarrow k$ be a partial function defined on a nonempty Zariski open subset of $X \times k^1$, if $f_a$ and $f^b$ are rational functions for all (a,b) $\in \operatorname{dom}f$, then $f$ is a rational function.
\end{theorem}
\begin{proof}
    If $\left.f\right|_{\operatorname{dom}f}\equiv 0$ then the statement holds trivially. For the subsequent proof we shall assume that $f$ is not constant zero on its domain of definition. Therefore up to a translation we can assume w.l.o.g. that $f^0\colon X \dashrightarrow k$ is a nonzero rational function. By virtue of Corollary \ref{corollary3_2}, we can assume w.l.o.g that $X=\left\{a \in \operatorname{dom} f^0 \mid f(a, 0) \neq 0\right\}$. In particular, we have that $f_a\colon k \dashrightarrow k$ is a nonzero rational function for each $a \in X$.

    For any $d \in \mathbb{N}$ and $e \in \mathbb{Z}$ define $\Lambda_d^e:=\left\{a \in X \mid \operatorname{deg}f_a=d\right.$, $\left.\operatorname{ord}_{\infty}\left(f_a\right)=e\right\}$, then $X=\coprod_{e \in \mathbb{Z}} \coprod_{d=0}^{\infty} \Lambda_d^e$. By Proposition \ref{proposition4_3} there exists $(d, e) \in \mathbb{N} \times \mathbb{Z}$ and a Zariski dense subset $\Lambda$ of $X$, such that $\operatorname{deg} f_a=d$ and $\operatorname{ord}_{\infty}\left(f_a\right)=e$ for all $a \in \Lambda$, and the cardinality of $\Lambda$ is continuum. Denote $n:=d+\min \{0, l\}, m:=d-\max \{0, l\}$ and $l:=n+m$, then by Proposition \ref{proposition4_5} we have $n, m, l \in \mathbb{N}$.

    Since $\left\{\operatorname{dom} f_a\right\}_{a \in \Lambda}$ is a continuum family of cofinite subsets of $k^1$ and the cardinality of $k$ exceeds continuum, we conelude that $S:=\bigcap_{a \in \Lambda}\operatorname{dom} f_a$ is infinite, and hence Zariski dense in $k^1$. In particular, there exists distinct elements $b_0, \dots, b_{l} \in S$. Again by virtue of Corollary \ref{corollary3_2}, we can assume w.l.o.g. that $X=\bigcap_{i=0}^{l} \operatorname{dom} f^{b_i}$.

    Define the regular functions $\phi\colon X \times k^1 \rightarrow k$ and $\psi\colon X \times k^1 \rightarrow k$ by
    $$
    \begin{aligned}
        \phi(a, b):=&\left|\begin{array}{cccccccc}
            1 & b & \dots & b^n & 0 & 0 & \dots & 0\\ 
            1& b_0 & \dots & b_0^n & f\left(a, b_0\right) & f\left(a, b_0\right) b_0 & \dots & f\left(a, b_0\right) b_0^m \\ 
            \vdots & \vdots & \ddots & \vdots & \vdots & \vdots & \ddots & \vdots \\ 
            1& b_l & \dots & b_l^n & f\left(a, b_l\right) & f\left(a, b_l\right) b_l & \dots & f\left(a, b_l\right) b_l^m
        \end{array}\right|\\ 
        \psi(a, b):=&\left|\begin{array}{cccccccc}
            1 & b & \dots & b^m & 0 & 0 & \dots & 0\\ 
            f\left(a, b_0\right) & f\left(a, b_0\right)b_0 & \dots & f\left(a, b_0\right)b_0^m & 1 & b_0 & \dots & b_0^n \\ 
            \vdots & \vdots & \ddots & \vdots & \vdots & \vdots & \ddots & \vdots \\ 
            f\left(a, b_l\right) & f\left(a, b_l\right)b_l & \dots & f\left(a, b_l\right)b_l^m & 1 & b_l & \dots & b_l^n
        \end{array}\right|
    \end{aligned}
    $$

    Take any $a \in \Lambda$, then by Proposition \ref{proposition4_5} $\psi_a(b) \neq 0$ and $f_a(b)=(-1)^{n m} \phi_a(b) / \psi_a(b)$ for all $b \in S \subseteq \operatorname{dom} f_a$.

    Define rational function $g\colon X \times k \dashrightarrow k$ by $g(a, b):=\dfrac{\phi(a, b)}{\psi(a, b)}$ on $\operatorname{dom}g:=\left\{(a, b) \in \operatorname{dom}f \mid \psi_a(b) \neq 0\right\}$, then $\Lambda \times S \subseteq \operatorname{dom}g$ and by construction $f(a, b)=g(a, b)$ for all $(a, b) \in \Lambda \times S$. Recall that $\Lambda, S$ are Zariski dense subsets of $X, k^1$ respectively, the desired result then follows from Proposition \ref{proposition3_3}.
\end{proof}
\begin{corollary}\label{corollary4_7}
    Let $k$ be an algebraically closed field of cardinality exceeding continuum, and let $f\colon k^n \dashrightarrow k$ be a partial function defined on a nonempty Zariski open subset of $k^n$, if $f_{a_1, \dots, a_{i-1}}^{a_{i+1}, \dots, a_n}\colon k \dashrightarrow k$ is a rational function for every $\left(a_1, \dots, a_n\right)\in \operatorname{dom}f$ and $i \in\{1, \dots, n\}$, then $f$ is a rational function.
\end{corollary}
\begin{proof}
    Recall Theorem \ref{theorem4_6} and apply mathematical induction on $n \in \mathbb{N}$.
\end{proof}

\quad

University of Warwick
Coventry, CV4 7AL, UK

hanwen.liu@warwick.ac.uk

\end{document}